\UseRawInputEncoding

\documentclass[12pt]{amsart}
\usepackage{amsfonts}
\usepackage{amsmath}
\usepackage{amssymb,latexsym}
\usepackage[mathcal]{eucal}
\usepackage[pdftex,bookmarks,colorlinks,breaklinks]{hyperref}

\input xy
\xyoption{all}

\newcommand{\Ccal}{\mathcal{C}}

\newcommand{\Ecal}{\mathcal{E}}

\newcommand{\Gcal}{\mathcal{G}}

\newcommand{\Kcal}{\mathcal{K}}

\newcommand{\Pcal}{\mathcal{P}}
\newcommand{\Qcal}{\mathcal{Q}}
\newcommand{\Rcal}{\mathcal{R}}
\newcommand{\Scal}{\mathcal{S}}

\newcommand{\Ucal}{\mathcal{U}}

\newcommand{\Xcal}{\mathcal{X}}
\newcommand{\Ycal}{\mathcal{Y}}
\newcommand{\Zcal}{\mathcal{Z}}

\newcommand{\ch}{\mathbf{1}}

\newcommand{\Z}{\mathbb{Z}}

\newcommand{\N}{\mathbb{N}}

\newcommand{\Ab}{\mathbf{A}}
\newcommand{\Bb}{\mathbf{B}}

\newcommand{\Wb}{\mathbf{W}}
\newcommand{\Xb}{\mathbf{X}}
\newcommand{\Yb}{\mathbf{Y}}
\newcommand{\Zb}{\mathbf{Z}}

\newcommand{\al}{\alpha}

\newcommand{\del}{\delta}

\newcommand{\ep}{\epsilon}
\newcommand{\sig}{\sigma}

\newcommand{\la}{\lambda}

\newcommand{\om}{\omega}

\newcommand{\br}{\vspace{3 mm}}

\newcommand{\tri}{\bigtriangleup}
\newcommand{\rest}{\upharpoonright}

\newcommand{\id}{{\rm{id}}}
\newcommand{\Id}{{\rm{Id}}}

\newcommand{\dist}{{\rm{dist}}}

\swapnumbers
\theoremstyle{plain}
\newtheorem{thm}{Theorem}[section]

\newtheorem{lem}[thm]{Lemma}
\newtheorem{prop}[thm]{Proposition}

\theoremstyle{definition}
\newtheorem{defn}[thm]{Definition}

\newtheorem{rmk}[thm]{Remark}

\newtheorem{prob}[thm]{Problem}

\begin{document}


\title[An ergodic system is dominant exactly when it has positive entropy]
{An ergodic system is dominant exactly when it has positive entropy}

\author[Tim Austin]{Tim Austin}
\address{Department of Mathematics, University of California, Los Angeles, Los Angeles, CA 90095-1555, USA} \email{tim@math.ucla.edu}

\author[Eli Glasner]{Eli Glasner}
\address{Department of Mathematics,
Tel-Aviv University, Ramat Aviv, Israel}
\email{glasner@math.tau.ac.il}

\author[Jean-Paul Thouvenot]
{Jean-Paul Thouvenot}
\address{Laboratoire de Probabilit\'es, Universit\'e 
Statistique et Mod\'{e}lisation, Sorbonne Universit\'{e},
4 Place Jussieu, 75252 Paris Cedex 05, France}
\email{jean-paul.thouvenot@upmc.fr}

\author[Benjamin Weiss]{Benjamin Weiss}
\address{Mathematics Institute, Hebrew University of Jerusalem,
Jerusalem, Israel}
\email{weiss@math.huji.ac.il}

%
%
%
%
%
%


\setcounter{secnumdepth}{2}



\setcounter{section}{0}


 
%

\begin{abstract}
An ergodic dynamical system $\Xb$ is called dominant if it is isomorphic to a 
generic extension of itself. It was shown in \cite{GTW} that Bernoulli systems with finite entropy
are dominant. In this work we show first that every ergodic system with positive entropy  
 is dominant, and then that if $\Xb$ has zero entropy then it is not dominant.
\end{abstract}

\keywords{dominant systems, generic properties, Bernoulli systems, relative Bernoulli,
very weak Bernoulli}

\thanks{{
\em 2010 MSC2010:
37A25, 37A05, 37A15, 37A20}}

\begin{date}
{December 7, 2021}
\end{date}

\maketitle


\tableofcontents
\setcounter{secnumdepth}{2}


\setcounter{section}{0}


\section*{Introduction}
We say that an ergodic system $\Xb = (X, \Xcal, \mu, T)$ is {\em dominant} if a generic extension
$\hat{T}$ of $T$ is isomorphic to $T$.   We obtain the surprising result that every ergodic positive entropy
system of an amenable group has the property that its generic extension is isomorphic
to it. For $\Z$ systems we show that conversely, when an ergodic system has zero
entropy then it is not dominant. Our first result  for $\Z$ actions follows from an extension of a result from 
\cite{GTW} according to which a generic extension of a Bernoulli system is Bernoulli with the same entropy (and hence is isomorphic to it by Ornstein's fundamental result) to the relative situation - together with 
Austin's weak Pinsker theorem \cite{Au}.
 The extension to all countable amenable groups relies on the results in 
\cite{OW-80}, \cite{R-W} and \cite{D-P}. For the result that zero
 entropy is not dominant for $\Z$ actions we use an idea from the slow entropy developed in \cite{K-T}.

\br

To make the definition of dominance more precise,
as in \cite{G-W} and \cite{GTW}, we 
present a convenient way of parametrising the space of extensions of $T$ as follows:
Let $\Xb = (X, \mathcal{X},\mu,T)$ be an ergodic system.
We will assume throughout this work (excepting the last section, where we will
comment about the infinite entropy case) that it is infinite and has finite entropy,
which for convenience we assume is equal $1$. 
Let $\Rcal \subset \Xcal$ be a finite generating partition.
Let $\Scal$ be the collection of Rokhlin cocycles with values in the Polish group
of measure preserving automorphisms of
the unit interval MPT$(I, \Ccal, \la)$, where 
$\la$ is the normalized Lebesgue measure 
and $\Ccal$ is the Borel $\sig$-algebra on $I = [0,1]$.
Thus an element $S \in \Scal$ is a measurable map $x \mapsto S_x \in $ MPT$(I, \la)$,
and we associate to it the {\em skew product transformation}
$$
\hat{S}(x,u) = (Tx, S_x u),\quad  (x \in X, u \in I),
$$ 
on the measure space $(X \times I, \Xcal \times \Ccal, \mu \times \la)$.

We recall that, by Rokhlin's theorem,  every ergodic extension $\Yb \to \Xb$ either has this form
or it is $n$ to $1$ a.e for some $n \in \N$ (see e.g. \cite[Theorem 3.18]{G}).
Thus the collection $\Scal$ parametrises the ergodic extensions of $\Xb$ with infinite fibers.
This defines a Polish topology on $\Scal$ which is inherited from 
the Polish group MPT$(X \times I, \mu \times \la)$ of all the measure preserving transformations.

In \cite{GTW} we have shown that for a fixed ergodic finite entropy $T$ with property $\Ab$, 
a generic extension $\hat{T}$ of $T$ also has the property $\Ab$,
where $\Ab$ stands for each of the following properties:
(i) having the same entropy as $T$, (ii) Bernoulli, (iii)  K, and (iv) loosely Bernoulli.

Now with these notations at hand the definition above becomes:

\begin{defn}
An ergodic system $\Xb = (X, \mathcal{X},\mu,T)$ is {\em dominant} if 
there is a dense $G_\del$ subset $\Scal_0 \subset \Scal$ such that for each
$S \in \Scal_0$ we have $\hat{S} \cong T$.
\end{defn}

From  \cite[Theorems 4.1 and 5.1]{GTW}, if $\Bb$ is a Bernoulli system with 
finite entropy, then its generic extension is again Bernoulli having the same entropy.
By Ornstein's theorem  \cite{O} such an extension is isomorphic to $\Bb$.
This proves the following.

\begin{prop}
Every Bernoulli system with finite entropy is dominant.
\end{prop}

We recall that an ergodic system $\Xb$  is {\em coalescent} is every endomorphism 
$E$ of $\Xb$ is an automorphism. 
Note that when an extension $\hat{S}$ as above with $ \hat{S} \cong T$ exists then 
the the system $\Xb$ is not coalescent. In fact, if $\pi : \hat{S} \to T$ is the 
(infinite to one) extension, and $\theta : T \to \hat{S}$ is an isomorphism then
$E = \pi \circ \theta$ is an endomorphism of $\Xb$ which is not an automorphism.
Thus we have the following:

\begin{prop}\label{dnc}
A dominant system is not coalescent.
\end{prop}
 
Hahn and Parry  \cite{H-P}  showed that totally ergodic automorphisms with 
quasi-discrete spectrum are coalescent.
In \cite{N} Dan Newton says:
\begin{quote}
A question put to me by Parry in conversation is the following: if $T$ has
positive entropy does it follow that $T$ is not coalescent ?
\end{quote}

Using theorems of Ornstein \cite{O} and Austin \cite{Au}, 
and results from \cite{GTW} we can now prove the following theorem.

\begin{thm}
An ergodic system with positive entropy is not coalescent.
\end{thm}

\begin{proof}
We first observe that a Bernoulli system is never coalescent 
(if $\Bb$ is Bernoulli and $\Bb' \to \Bb$ is an isometric extension which is 
again Bernoulli then, by Ornstein's theorem, $\Bb' \cong \Bb$).
Now let $\Xb = (X, \mathcal{X},\mu,T)$ be an ergodic system with positive entropy.
By Austin's weak Pinsker theorem \cite{Au} we can write $\Xb$ as a product system
$\Bb \times \Zb$ with $\Bb$ a Bernoulli system of finite entropy. 
Finally, as noted in \cite[Proposition 1]{N}
if $T = T_1 \times T_2$, where $T_1$ is not coalescent, then $T$ is not coalescent.
In fact, given an endomorphism $E$ of $T_1$ which
is not an automorphism, the map $E \times \Id$, where $\Id$ denotes the identity automorphism on the second coordinate, is an endomorphism of $T$ which is not an automorphism.
Applying this observation to $\Xb = \Bb \times \Zb$ we obtain our claim.
\end{proof}

These results suggest the following question:
is every ergodic system of zero entropy not dominant ?
At least generically we immediately see that the answer is affirmative.
As was shown in \cite{N} the set of coalescent automorphisms in MPT$(I, \la)$
is comeager. Thus by Proposition \ref{dnc} we conclude that:
the set of non-dominant automorphisms is comeager in MPT$(I, \la)$, hence also
in the dense $G_\del$ subset of MPT$(I, \la)$ comprising the zero entropy automorphisms.
However, as we will show in section \ref{0entropy} using a slow entropy argument,
the answer is affirmative for every ergodic system with zero entropy.

\br

\section{Relative Bernoulli}\label{Sec-rB}

\begin{defn}\label{d:rB}
Let $\Xb = (X, \mathcal{X},\mu,T)$ be an ergodic system and $\Xcal_0 \subset \Xcal$ a $T$-invariant
$\sig$-subalgebra. Let $\Xb_0 = (X_0, \mathcal{X}_0,\mu_0,T_0)$ be the corresponding factor system
and let $\pi : \Xb \to \Xb_0$ denote the factor map.
We say that $\Xb$ is {\em relatively Bernoulli} over $\Xb_0$ if there is a 
$T$-invariant $\sig$-algebra $\Xcal_1 \subset \Xcal$ independent of
$\Xcal_0$ such that $\Xcal = \Xcal_0 \vee \Xcal_1$, 
and there is a $\Xcal_1$-generating finite partition $\Kcal \subset \Xcal_1$
such that the partitions  $\{T^i \Kcal\}_{i \in \Z}$ are independent; in other words,
the corresponding system $\Xb_1 = (X_1, \mathcal{X}_1,\mu_1,T_1)$ is Bernoulli
and $\Xb \cong \Xb_0 \times \Xb_1$.
\end{defn}

If $\Rcal_0$ is a finite generating partition for $\Xcal_0$ and 
$\Rcal$ is a finite generating partition for $\Xcal$ then J.-P.  Thouvenot showed 
 that there is a condition called relatively weak Bernoulli, which is equivalent to the extension being
 relatively Bernoulli, see \cite{Th-75b} and also \cite{Ki}.  This condition is as follows:
 
 \begin{defn}
 The partition $(\Rcal,T)$ is {\em relatively Bernoulli} over $(\Rcal_0,T)$ if for every 
 $\ep >0$ there is $N$ such that for a collection $\Gcal$ of atoms $A$ of the partition
 $\vee_{i=-\infty}^{-1} T^{-i} \Rcal$,
 and  a collection $\Gcal_0$ of atoms $B$ of the partition
 $\vee_{i=-\infty}^{-\infty} T^{-i} \Rcal_0$, 
we have 
\begin{subequations}\label{*}
\begin{equation}\label{*a}
  \mu\left(\bigcup \{ A \cap B : A \in \Gcal, B \in \Gcal_0\}\right) > 1 - \ep,
 \end{equation}
 \begin{equation} \label{*b}
 \bar{d}_N (
 \dist (\vee_{i=0}^{N-1} T^{-i} \Rcal \rest A \cap B), 
 \dist (\vee_{i=0}^{N-1} T^{-i} \Rcal \rest  B))  < \ep, 
 \end{equation}
\end{subequations}
for all such $A$ and $B$.
 \end{defn}
 
Since $\vee_{i=-k}^{-1} T^{-i} \Rcal \nearrow \vee_{i=-\infty}^{-1} T^{-i} \Rcal$ and
$\vee_{i=-k}^{k} T^{-i} \Rcal_0 \nearrow \vee_{i=-\infty}^{\infty} T^{-i} \Rcal_0$, this can be 
formulated in finite terms as :
for every $\ep >0$ there is $N$ and
$\exists k_0$ such that for all $k > k_0$ there is a collection $\Gcal$ of atoms $A$ of 
 $\vee_{i=-k}^{-1} T^{-i} \Rcal$ 
 and a collection $\Gcal_0$ of atoms $B$ of 
$\vee_{i=-k}^{k} T^{-i} \Rcal_0$ such that
\begin{subequations}\label{**}
\begin{equation}\label{**a}
\mu\left(\bigcup \{ A \cap B : A \in \Gcal, B \in \Gcal_0\}\right) > 1 - \ep,
 \end{equation}
 \begin{equation} \label{**b}
 \bar{d}_N (
 \dist (\vee_{i=0}^{N-1} T^{-i} \Rcal \rest A \cap B), 
 \dist (\vee_{i=0}^{N-1} T^{-i} \Rcal \rest  B))  < \ep,
 \end{equation}
\end{subequations}
for all such $A$ and $B$.

 One last change --- instead of (\ref{**b}) we can also require that for $A, A' \in \Gcal, \ B \in \Gcal_0$
\begin{equation}\label{***}
 \bar{d}_N (
  \dist (\vee_{i=0}^{N-1} T^{-i} \Rcal \rest A \cap B), 
  \dist (\vee_{i=0}^{N-1} T^{-i} \Rcal \rest  A' \cap B))  < \ep. 
\end{equation}
That (\ref{**b}) implies (\ref{***}) with $2\ep$ is immediate.

For the converse implication observe first that 
the  distribution
 $ \dist (\vee_{i=0}^{N-1} T^{-i} \Rcal \rest B)$ is the average of  
 $ \dist (\vee_{i=0}^{N-1} T^{-i} \Rcal \rest A \cap B)$ over all $A \in \vee_{i=-k}^{k} T^{-i} \Rcal$,
  and that the $\bar{d}$ metric is a convex function of distributions.
Therefore fixing one $A' \in \Gcal$ and averaging over all $A \in \Gcal$ we get (\ref{**b}).   
  
  \br
  
\section{Positive entropy systems are dominant}\label{positive-entropy}

The next theorem is a relative version of Theorem 5.1 in \cite{GTW}
and serves as the main tool in the proof of Theorem \ref{D} below.

\begin{thm}\label{Thm-rB}
Let $\Xb = (X, \Xcal, \mu,T)$ be an ergodic system which is relative Bernoulli over $\Xb_0$ 
with finite relative entropy, so that $ \Xb = \Xb_0 \times \Xb_1$.
Then, the generic extension $\hat{S}$ of $T$ is relatively Bernoulli over $\Xb_0$.
\end{thm}

\begin{proof}
For convenience we assume that the relative entropy is $1$.

As in \cite{GTW}
let $\Rcal \subset \Xcal$ be 
a finite relatively generating partition for $\Xb$ over $\Xb_0$ with entropy $1$
(so that $\Rcal$ is  a Bernoulli partition  independent of $\Xb_0$),
and let $\Rcal_0 \subset \Xcal_0$ be a finite generator for $\Xb_0$.
Let $\Scal$ be the collection of Rokhlin cocycles with values in MPT$(I, \la)$, where 
$\la$ is the normalized Lebesgue measure on the unit interval $I =[0,1]$.
Thus an element $S \in \Scal$ is a measurable map $x \mapsto S_x \in $ MPT$(I, \la)$,
and we associate to it the {\em skew product transformation}
$$
\hat{S}(x,u) = (Tx, S_x u),\quad  (x \in X, u \in I).
$$ 
Let $Y = X \times I$ and set $\Yb = (Y, \mathcal{Y}, \mu \times \la)$, with $\Ycal = \Xcal \otimes \Ccal$.

\br

{\bf Part I}:
By Theorem 4.1 of \cite{GTW} there is a 
dense $G_\del$ subset $\Scal_0 \subset \Scal$ with
$h(\hat{S}) = 1$ for every $S \in \Scal_0$.
We will first show that the collection of
the elements $S \in \Scal_0$ for which the corresponding $\hat{S}$ is 
relatively Bernoulli over $\Xb_0$ forms a $G_\del$ set.

As the inverse limit of relatively Bernoulli systems is relatively Bernoulli, see \cite[Proposition 7]{Th-75a},
to show that a transformation $T$ on $(X, \Xcal, \mu)$ is relatively 
Bernoulli over $\Xb_0$ it suffices to show that for a refining sequence of partitions
$$
\Pcal_1 \prec \cdots \prec \Pcal_n \prec \Pcal_{n+1} \prec \cdots
$$
such that the corresponding algebras $\hat{\Pcal}_n$ satisfy
$\bigvee_{n \in \N} \hat{\Pcal}_n = \Xcal$, for each $n$, the process $(T, \Pcal_n)$ is 
relatively VWB relative to $(T,\Rcal_0)$.

For each $n \in \N$ let $\Qcal_n$ denote the dyadic partition of $[0,1]$
into intervals of size $1/2^n$, and let 
$$
\Pcal_n = \Rcal \times \Qcal_n.
$$
For any $S \in \Scal_0$ the relative entropy of $\Yb = \Xb \times [0,1]$ over $\Xb_0$ is also $1$. 
Thus for all $n$ we have
$$
H(\Pcal_n \mid (\vee_{i= -\infty}^{-1}\hat{S}^{-i}\Pcal_n) \vee (\vee_{i= -\infty}^{\infty}\hat{S}^{-i}\Rcal_0)) =1,
$$
and for all $N \geq 1$
$$
H(\vee_{i=0}^{N-1}   \hat{S}^{-i}\Pcal_n \mid (\vee_{i= -\infty}^{-1}\hat{S}^{-i}\Pcal_n)  \vee 
(\vee_{i= -\infty}^{\infty}\hat{S}^{-i}\Rcal_0)) =N.
$$
Therefore, we can find a suitably small $\del >0$ such that for $k_0$ large enough
$$
H(\vee_{i=0}^{N-1}   \hat{S}^{-i}\Pcal_n \mid (\vee_{i= -k_0}^{-1}\hat{S}^{-i}\Pcal_n)  \vee 
(\vee_{i= -k_0}^{k_0}\hat{S}^{-i}\Rcal_0))  < N +\del.
$$

Now, conditioned on the partition
$$
 (\vee_{i= -k_0}^{-1}\hat{S}^{-i}\Pcal_n)  \vee (\vee_{i= -k_0}^{k_0}\hat{S}^{-i}\Rcal_0)
$$
the partition $\vee_{i=0}^{N-1}\hat{S}^{-i}\Pcal_n$ will be $\eta$-independent of
$$
(\vee_{i= -k}^{-k_0 -1}\hat{S}^{-i}\Pcal_n)  \vee (\vee_{i= -k}^{-k_0 +1}\hat{S}^{-i}\Rcal_0)
\vee (\vee_{i= k_0 +1}^{k} \hat{S}^{-i}\Rcal_0)
$$
for all $k \geq k_0$
for $\eta$ small enough 
(see Definition 5.1 in \cite{GTW} and the following discussion), so that 
the inequality 
(\ref{***}) in Section \ref{Sec-rB} (with $\Pcal_n$ replacing $\Rcal$)
for $k = k_0$ will imply (\ref{***}) with $2\ep$, for all $k > k_0$.

\br 

Define the set $U(n, N_1, N_2, \ep, \del)$ to consist of those $S \in \Scal_0$ that satisfy:
\begin{enumerate}
\item
$H(\vee_{i=0}^{N_1 -1} \hat{S}^{-i} \Pcal_n \mid (\vee_{i=-N_2}^{-1} \hat{S}^{-i} \Pcal_n) 
\vee (\vee_{i=-N_2}^{N_2} \hat{S}^{-i} \Rcal_0   )) < N_1 + \del$,
\item
\begin{align*}
& \bar{d}_{N_1} \left(\vee_{i=0}^{N_1 -1} \hat{S}^{-i} \Pcal_n \rest A \cap B,  
\vee_{i=0}^{N_1 -1} \hat{S}^{-i} \Pcal_n \rest A' \cap B \right)  < \ep, \\
&  {\text{for a set of atoms}} \  A, A' \in \Gcal, \ B \in \Gcal_0, \\
& \ {\text{where}}\  
\Gcal \subset \vee_{-N_2}^{-1}\hat{S}^{-i} \Pcal_n, 
\ \Gcal_0 \subset \vee_{-N_2}^{N_2} \hat{S}^{-i} \Rcal_0 \\
&  {\text{and}}\ 
(\mu \times \la)\left(\bigcup \{A\cap B : A \in \Gcal, \ B \in \Gcal_0\} \right) > 1 - \ep.
\end{align*}
\end{enumerate}
Now the sets $U(n, N_1, N_2, \ep, \del)$ are open (easy to check) and that the $G_\del$ set
$$
\Scal_1 = \bigcap_{n, k, l} \bigcup_{N_1, N_2} U(n, N_1, N_2, 1/k, 1/l)
$$
comprises exactly the elements $S \in \Scal_0$ for which the corresponding $\hat{S}$ is 
relatively Bernoulli over $\Xb_0$. 
Thus, if $S \in \Scal_0$ is such that $\hat{S}$ is relatively  Bernoulli, then for every $n, \ep, \del$, there are
$N_1, N_2$ such that $S \in U(n, N_1, N_2, \ep, \del)$, and conversely, for every 
relatively  Bernoulli $\hat{S}$ the corresponding $S$ is in $\Scal_1$.

\br 

{\bf Part II}:
The collection $\Scal_1$ is nonempty. To see this we first note that
the Bernoulli system $\Xb_1$ admits a proper extension 
$\hat{\Xb}_1 \to \Xb_1$ which is also  Bernoulli and has the same entropy.
This follows e.g. by a deep result of Rudolph \cite{Ru-79, Ru-85} who showed that
every weakly mixing group extension of $\Xb_1$ is again a Bernoulli system.
An explicit example of such an extension of the $2$-shift is given by Adler and Shields, \cite{A-S}.
Since $\hat{\Xb}_1$ is weakly mixing the product system 
$\hat{\Xb} = \Xb_0 \times \hat{\Xb}_1$ is ergodic and 
$\hat{\Xb} \to \Xb_0$ is an element of $\Scal_1$.

Now apply the relative Halmos theorem  \cite[Proposition 2.3]{G-W}, to deduce that the 
$G_\del$ subset $\Scal_1$ is dense in $\Scal$, as claimed. 
\end{proof}

\br

We can now deduce the positive entropy part of our main result.

\begin{thm}\label{D}
Every ergodic system $\Xb = (X, \Xcal, \mu,T)$ of positive finite entropy is dominant.
\end{thm}

\begin{proof}
By Austin's weak Pinsker theorem \cite{Au} we can present $\Xb$ as a product system
$\Xb = \Bb \times \Zb$, where $\Bb$ is a Bernoulli system with finite entropy.
Thus $\Xb$ is relatively Bernoulli over $\Zb$, and by Theorem \ref{Thm-rB} it follows that a generic extension 
$\hat{S}$ of $\Xb$ is relatively Bernoulli over $\Zb$. Therefore, for such $\hat{S}$ the system
$\Yb = (X \times I, \mathcal{X} \times \Ccal, \mu \times \la, \hat{S})$
is again of the form $\Yb = \Bb' \times \Zb$ with $\Bb'$ a Bernoulli system
with  the same entropy as that of $\Bb$.
By Ornstein's theorem \cite{O} $\Bb \cong \Bb'$, whence also $\Xb \cong \Yb$
and our proof is complete.
\end{proof}

\begin{rmk}
With notations as in the proofs of Theorems  \ref{Thm-rB} and  \ref{D} observe that
for every $S \in \Scal$ the system $(Y, \mu \times \la,\hat{S})$ admits $\Zb = (Z, \Zcal, \mu,T)$
(with $\Zcal$ considered as a subalgebra of $\Xcal$) as a factor:
$$
(Y, \mu \times \la,\hat{S}) \to \Xb \to \Zb.
$$
In the Polish group $G=$MPT$(Y,\mu \times \la)$ consider the closed subgroup
$G_{\Zb} = \{g \in G : gA =A ,\ \forall A \in\Zcal\}$.
We now observe that the residual set $\Scal_1 \subset \Scal_0$, of 
those $S \in \Scal_0$ for which $\hat{S}$ is Bernoulli over $\Zb$
with the same relative entropy over $\Xb$ 
is a single orbit for the action of $G_{\Zb}$ under conjugation.
\end{rmk}

In the last section (Section \ref{sec-A}) we will show that the positive entropy theorem 
holds for any countable amenable group.

\br

In  \cite[Theorem 6.4]{GTW}  it was shown that the generic extension of a 
K-automorphism is a mixing extension.
 We will next prove an analogous theorem for a general ergodic system
 with positive entropy.
We first prove the following relatively Bernoulli analogue of Theorem 6.2  in \cite{GTW}. 


\begin{thm}\label{Thm-rBB}
Let $\Xb = (X,\Xcal,\mu,T)$ be a relatively Bernoulli system over  $\Xb_0$, and $S$ a Rokhlin cocycle
with values in MPT$(I, \la)$, where $I =[0,1]$ and $\la$ is Lebesgue measure on $I$.
We denote by $\hat{S}$ the transformation 
$$
\hat{S}(x,u) = (Tx, S_xu),
$$
on $Y = X \times I$, and let
$$
\check{S} 
(x, u, v) = (Tx, S_xu, S_xv), \quad (x,u,v) \in W = X \times I \times I,
$$
be the relative independent product of $\Yb$ with itself 
over $\Xb$. 
Then for a generic $S \in \Scal$ the transformation $\check{S}$ is relatively Bernoulli over $\Xb_0$. 
\end{thm}

\begin{proof}
For the $G_\del$ part we follow, almost verbatim, the proof of Theorem  \ref{Thm-rB},
where we now let $\Qcal_n$ denote the product dyadic partition of $I \times I$ into squares of size 
$\frac{1}{2^n} \times \frac{1}{2^n}$ and,
with notations as in the proof of
Theorem  \ref{Thm-rB}, we let $\Pcal_n = \Rcal \times \Qcal_n$.

Thus it only remains to show that the $G_\del$ set $\Scal_1$, comprising those 
$S \in \Scal_0$ for which $\check{S}$ is relatively Bernoulli on $W = X \times I \times I$ relative to
$\Xb_0$, is non empty.
Now examples of skew products over a Bernoulli system with such properties
 are  provided by Hoffman in \cite{Ho}.
The base Bernoulli transformation that Hoffman constructs for his example can be 
arranged to have arbitrarily small entropy 
by an appropriate choice of the parameters 
used in the construction in section 3
(the skew product example is in section 4 and the proof of Bernoullicity is in section 5).
Using such construction on $\Xb$ (where the cocycle is measurable
with respect to the Bernoulli direct component of $\Xb$)
we obtain our required extension of $\Xb$.
This completes our proof.
\end{proof}

We also recall the following criterion \cite[Lemma 6.5]{GTW}:

\begin{lem}\label{6.5}
Let $\Xb$ be ergodic and $\Yb$ be a factor of $\Xb$. Then the following are equivalent:
\begin{enumerate}
\item
$\Xb$ is a relatively mixing extension of $\Yb$.
\item
In the relatively independent product $X\underset {Y}{\times} X$, the Koopman operator restricted to $L^2(Y)^{\perp}$ is mixing.
\end{enumerate}
\end{lem}

\br

\begin{thm}\label{MixD}
Let $\Xb =(X, \Xcal,\mu,T)$ be an ergodic system with positive entropy, then the generic extension of $\Xb$ 
is relatively mixing over $\Xb$.
\end{thm}

\begin{proof}
By the weak Pinsker theorem \cite{Au} we can present $\Xb$ as a product system
$\Xb = \Zb \times \Bb$, where $\Bb$ is a Bernoulli system with finite entropy.
Thus $\Xb$ is relatively Bernoulli over $\Zb$, 
and by Theorem \ref{Thm-rBB} it follows that a generic extension 
$\check{S}$ of $\Xb$ to $X \times I \times I$ is still relatively Bernoulli over $\Zb$. 
Thus the extended system $\Wb$ on $W = X \times I \times I$ with $\check{S}$ action, has the form 
$\Wb = \Zb \times \Bb'$ with $\Bb'$ again a Bernoulli system.

Now for the system $\Yb$, defined on $Y = X \times I$ by
$$
\hat{S}(x,u) = (Tx, S_xu),
$$
we have that the corresponding relative product system
$\Yb \underset{\Xb}{\times} \Yb$ is isomorphic to $\Wb$,
which is a Bernoulli extension of $\Zb$ and therefore, 
by  Lemma \ref{6.5}, a relatively mixing extension of $\Zb$.
A fortiori $\Yb \underset{\Xb}{\times} \Yb$ is a relatively mixing
extension of $\Xb$ and our proof is complete.
\end{proof}

\br

\section{Zero entropy systems are not dominant}\label{0entropy}

\begin{defn}\ 
\begin{itemize}
\item
For $\om, \om' \in \{0,1\}^n$ the {\em Hamming} 
(or {\em $\bar{d}$-distance)} is defined by
$$
\bar{d}(\om, \om') =\frac1n\#\{0 \le i <n : \om_i \not= \om'_i\}.
$$
\item
For two measurable partitions $Q =\{A_i\}_{i=1}^n , \ \hat{Q} = \{B_i\}_{i=1}^n$ 
of a measured space $(X,\mu)$, the distance $d(Q,\hat{Q})$ is defined by
$$
d(Q,\hat{Q})  = \frac12 \sum_{i=1}^n \mu (A_i \tri B_i). 
$$
\end{itemize}
\end{defn}

\br

\begin{thm}\label{thm-0e}
Every ergodic system $\Xb$ with zero entropy is not dominant.
\end{thm}

\begin{rmk}
Recently Terrence Adams \cite{Ad} has proved a somewhat analogous result in the setting of MPT, 
the group of all measure preserving transformations of the unit interval with Lebesgue measure. 
It is well known that generically a $T$ in MPT has zero entropy. 
What Adams shows is that for any preassigned growth rate for slow entropy, the generic transformation
has a complexity which exceeds that rate.  
In our proof of theorem \ref{thm-0e} we don't introduce a formal definition of slow entropy but 
its definition lies behind our lemma \ref{1-delta}. 
\end{rmk}

\begin{proof}
We first choose a strictly ergodic model $\Xb =(X, \Xcal, \mu_0,T)$
for our system which is a subshift of $\{0,1\}^\Z$. By the variational principle
this model will have zero topological entropy.
(To see that such a model exists, see for example \cite{DGS}, where this
fact can be deduced from property (b) on page 281 and Theorem 29.2 on page 301.)
Denote by $a_n$ the number of $n$-blocks in $X$, so that $a_n$ is sub-exponential.

For $x_0 \in X$ and $\Qcal=\{Q_0,Q_1\}$ a partition of $X$ let
$$
B_n(x_0,\ep) = \{x \in X : \bar{d}_n(Q_n(x), Q_n(x_0)) < \ep \},
$$
where for a point $x \in X$ and $n \geq 1$ we write
$$
Q_n(x) =\om_0\om_1\om_2\dots\om_{n-1}, \ {\text{when}}\ 
x \in \cap_{i=0}^{n-1} T^{-i}(Q_{\om_i}).
$$
\begin{lem}\label{1-delta}
For $\ep < \frac{1}{100}$ and $\del < \frac{1}{100}$ there is an $N$ such that for all
$n \geq N$, if $m$ is the minimal number such that
there are points $x_1, x_2, \dots, x_m$ with
$$
\mu_0(\cup_{i=1}^m B_n(x_i,\ep)) > 1 -\del,
$$
then $m \leq a_{2n}$.
\end{lem}

\begin{proof}
Denote by $\Pcal =\{P_1, P_2\}$ the partition of $X$ according to the $0$-th coordinate.
Given $\ep>0$ there is some $k_0$ and a partition $\hat{\Qcal}$ measurable with respect to
$\vee_{i=-k_0}^{k_0}T^i \Pcal$ such that
$$
d(\Qcal, \hat{\Qcal}) < \frac{\ep}{2}.
$$
By ergodicity there exists an $N$ such that  for $n \geq N$ there is a set $A \subset X$ with
$\mu_0(A) > 1 - \del$ with 
\begin{equation*}\label{Adelta}
\bar{d}_n(Q_n(x), \hat{Q}_n(x)) < \ep, \qquad \forall x \in A.
\end{equation*}

Let $\{\al_i\}_{i=1}^\ell$ be those atoms of $\vee_{i=-k_0}^{n + k_0}T^i  \Pcal$
such that $\al_i \cap A \not=\emptyset$, so that $\ell \leq a_{n + 2k_0 +1}$.

Choose $x_i \in \al_i \cap A, \ 1\leq i \leq \ell$. 
We claim that
$$
A \subset \cup_{i=1}^\ell B_n(x_i,\ep).
$$
For $x \in \cup_{i=1}^\ell \al_i$ we denote by $i(x)$ that index such that $x \in \al_{i(x)}$.
Now since $x$ and $x_{i(x)}$ are in $A$ we have
$$
\bar{d}_n(Q_n(x), \hat{Q}_n(x)) < \ep \quad {\text{and}} \quad
\bar{d}_n(Q_n(x), \hat{Q}_n(x)) < \ep.
$$
Since $x \in \al_{i(x)}$, $\hat{Q}_n(x) = Q_n(x)$.
Therefore
$$
\bar{d}_n(Q_n(x), Q_n(x_{i(x)})) < 2\ep,
$$
whence $x \in B_n(x_{i(x)}, \ep)$.
This proves our claim and we conclude that $m \leq \ell \leq a_{n +2k_0 +1}$.
Thus for sufficiently large $n$ we indeed get $m \leq a_{2n}$.
\end{proof}

\br

We will show that a generic extension of $T$ to
$(Y, \mu) = (X \times [0,1],\mu_0 \times \la)$, 
with $\la$  Lebesgue measure on $[0,1]$, is not isomorphic to $\Xb$. 
To do this we will show that for a generic extension $\hat{S}$ 
the partition $\Qcal$ of $Y$, defined by splitting $X \times [0,1]$ into
$\{Q_0, Q_1\} = \{X \times [0,\frac12], X \times [\frac12,1]\}$, 
will not satisfy the conclusion of this lemma.

\br

{\bf Notations:}
\begin{itemize}
\item
$\Scal$ is the Polish space comprising the measurable Rohklin cocycles 
$x \mapsto S_x  \in {\rm MPT}([0,1], \la)$.
\item
For $S \in \Scal$ let
$\hat{S}(x,u) =(Tx,S_xu)$.
\item
$Q_n^{\hat{S}}(y) =\om_0\om_1\om_2\dots\om_{n-1}$, where
$y \in \cap_{i=0}^{n-1} \hat{S}^{-i}(Q_{\om_i})$.
\item
$C(\hat{S}, n, \ep, \del) = \min\{k : \exists y_1,y_2, \dots,y_k \in Y, \ {\text {such that}}\ 
\mu(\cup_{i=1}^n B_n^{\hat{S}}(y_i,\ep)) > 1 -\del$.
\end{itemize}

Define now
$$
\Ucal(N, \ep, \del) =
\{S \in \Scal : \exists n \geq N\ {\text{such that}}\ 
C(\hat S,n, \ep, \del) > 2 a_{2n}\}.
$$
This is an open subset of $\Scal$
(see e.g. \cite{GTW} for similar claims). We will show that, for sufficiently small $\ep$ and $\del$,
 it is dense in $\Scal$.

First consider the case $S_0 = \id$.
Let $\eta >0$ be given and choose $M$ so that $\frac1M < \eta$.
Now build a Rohklin tower for $T$, with base $B_0$ and heights $mM > N$ and $mM + 1$ 
for a suitable $m$, filling all of $X$.
Let $B = B_0 \times [0,1]$ be the base of the corresponding tower in $(Y,\mu, \hat{S})$.
We modify $S_0 =\id$ only on the levels
$T^{jM-1}B_0$ for $1 \leq j \leq m$, so that the new $S$ will be within $\eta$ of $S_0$.
The $Q$-$M$ names of the points in $T^{jM-1}B$ are constant for all $0 \leq j < m$. 
We modify $S_0$ on the levels
$T^{jM-1}B$ so that we see all possible
$0$\,-$1$ names for the $M$-blocks as we move up the tower with equal measure.
A similar procedure is described as {\em independent cutting and stacking}
and is explained in details in section I.10.d in Paul Shields book \cite{Shi}.

\begin{lem}\label{L-U}
Any $B_{mM}(y,\ep)$ ball has measure at most
$2^{m(- \frac12 + H(2\ep, 1 -2\ep))}$.
\end{lem}

\begin{proof}
The $Q_{mM}$-names of points $y \in B$ are constant on blocks of length  $M$ and all 
sequences of zeros and ones have equal probability by construction.
So by a well known estimation (using Stirling's formula), in $\{0,1\}^m$ with uniform measure, 
the measure of an $\ep$-ball in normalized Hamming metric is 
$\leq 2^{m(- \frac12 + H(2\ep, 1 -2\ep))}$.

For points in the lower half of the tower over $B$ we have
a similar estimate with $m$ replaced by some $\ell > \frac12 m$ 
and $\ep$ replaced by $\frac{m}{\ell} \ep < 2\ep$.
For points in the upper half of the tower, for some $\ell < \frac12 m$ 
we have that $\hat{S}^\ell y \in B$
and then we get an estimate with $m-\ell > \frac12 m$. 
This proves the lemma.
\end{proof}

From this lemma it follows that in order to achieve even $\frac12$ as
$\mu(\cup_{i=1}^L B_{mM}(y_i,\ep))$
we must have $L \cdot 2^{m(- \frac12 + H(2\ep, 1 -2\ep))} > \frac12$, hence
$$
L \geq \frac{1}{2}  \cdot 2^{m( \frac12 - H(2\ep, 1 -2\ep))}.
$$ 
Since $a_n$ is sub-exponential this lower bound certainly exceeds $a_{2mM}$ if $m$ is
sufficiently large.
This shows that this modified $S$ is an element of $\Ucal(N,\ep,\del)$.

\br

A similar construction can be carried out for any $S \in \Scal$.
The main point that needs to be checked is that for small $\ep$, no
$B_M^{\hat{S}}(y,\ep)$-ball can have measure greater than $\frac12 + 2\ep$,

\begin{lem}\label{1/2}
For any $\hat{S}$ and all $y_0$
$$
\mu(B_M^{\hat{S}}(y_0,\ep)) \leq \frac12 + 2\ep.
$$
\end{lem}

\begin{proof}
Let $Q_M^{\hat{S}}(y_0) = \om_0\om_1\dots\om_{M-1}$. Then
$$
\bar{d}_M(Q_M^{\hat{S}}(y), Q_M^{\hat{S}}(y_0)) = 
\frac1M \sum_{i=0}^{M-1} \ch_{Q_{\om_i}}(\hat{S}^iy_0)(1 - \ch_{Q_{\om_i}}(\hat{S}^iy)),
$$
and 
$$
\int_Y \bar{d}_M(Q_M^{\hat{S}}(y), Q_M^{\hat{S}}(y_0)) \, d\mu = \frac12.
$$
Since $\bar{d}_M \leq 1$, the measure of the set where
$ \bar{d}(Q_M^{\hat{S}}(y), Q_M^{\hat{S}}(y_0)) \leq \ep$ cannot exceed $\frac12 + 2\ep$.
\end{proof}

This lemma, which is formulated for the measure $\mu$ on the entire space $Y$, in fact holds
as well for any level $L_j = \hat{S}^{jM}B$ in the tower, when we replace $\mu$ by
the measure $\mu$ restricted to $L_j$.
This is so because the partition $\{Q_0, Q_1\}$ intersects each level of the tower
 in relative measure $\frac12$ and $\hat{S}$ is measure preserving.
 
We now mimic the proof outlined for $S_0 = \id$ and, given $S \in \Scal$, using
an independent cutting and stacking we change $\hat{S}$ as follows.
For the level $L_j = \hat{S}^{jM}B$ consider the partition 
$$
\Rcal_j = \vee_{i=0}^{M-1} \hat{S}^{-i}(\Qcal \cap \hat{S}^{jM +i}B).
$$
We change the transformation $\hat{S}$ at the transition from level $jM-1$ to level $jM$,
so that these partitions $\Rcal_j$ will become independent.

We want to estimate the size of an $mM$-$\epsilon$ ball around a point $y_0 \in B$.
If $y \in B$ belongs to this ball there is a set 
$A \subset \{0,1,2,...,nM-1\}$ with $|A| \leq \epsilon \, mM$
where the $mM$-names of $y$ and $y_0$ differ. 
We need now a simple lemma:

\begin{lem}\label{sqrt}
Let $A \subset \{0,1,\dots,mM-1\}$ such that
$|A| \leq \ep \, mM$.
Denote $I_j = \{jM, jM+1,\dots, jM+M -1\}, \ 0 \leq j < m-1$.
Let $J \subset \{0,1,\dots ,m-1\}$ be the set of $\ell$ such that 
$$
|I_\ell \cap A| < \sqrt{\ep} M.
$$
Then  $|J| >  (1-\sqrt{\ep}) m$.
\end{lem}

\begin{proof}
Let $K = \{0,1,\dots,mM-1\} \setminus J$.
Then
$$
\ep \, mM \geq | \cup_{k \in K} I_k \cap A| \geq M \sqrt{\ep} |K|.
$$
Thus $|K| \leq \sqrt{\ep} m$, whence
 $|J| >  (1-\sqrt{\ep}) m$.
 \end{proof}

Next using Lemma \ref{1/2} for each level of the form $T^{jM}B_0$, we will estimate
the size of an $mM$-$\ep$ ball.
So fix a point $y_0 \in B$. If $y \in B_{mM}(y_0,\ep)$ then 
by Lemma \ref{sqrt} there is a set of indices $J_y \subset \{1.2.\dots,,m\}$ such that
\begin{enumerate}
\item
$|J_y| \geq (1 -\sqrt{\ep})m$,
\item
For each $j \in J_y$, 
$\hat{S}^{jM}y \in B_M(\hat{S}^{jM}y_0, \sqrt{\ep})$.
\end{enumerate}
The number of possible sets that satisfy (1) is bounded by
$2^{mH(\sqrt{\ep}, 1- \sqrt{\ep})}$.
By Lemma \ref{1/2} and by the independence, for  a fixed such $J_y$ 
the measure of the set of points that satisfy (2) is at most
$$
(\frac12 + 2\sqrt{\ep})^{m(1-\sqrt{\ep})}.
$$
Write $(\frac12 + 2\sqrt{\ep})^{1-\sqrt{\ep}} = 2^{-c}$, where $c \geq c_0 >0$
for all sufficiently small $\ep$.
Then
$$
2^{-cm} \cdot 2^{m H(2\ep, 1 -2\ep)} = 2^{m(-c + H(2\ep, 1 -2\ep))} \leq 2^{-\frac{m}{2} c_0},
$$
for $H(2\ep, 1 -2\ep) \leq \frac12 c_0$.
We now see that the measure of the ball $B_{mM}(y_0,\ep)$ is bounded 
by $2^{-\frac{m}{2} c_0}$.

This was done for $y_0 \in B$ and as in the proof of Lemma \ref{L-U} we obtain
the suitable estimations for any $y$ in the tower over $B$.
We conclude the argument as in the case $S= \id$ and
again it follows that  the resultant modified $S$ is an element of $\Ucal(N,\ep,\del)$.

Finally for fixed sufficiently small $\ep$ and $\del$ setting
$$
\Ecal = \bigcap_{N=1}^\infty   \, \Ucal(N, \ep, \del)
$$
we obtain the required dense $G_\del$ subset of $\Scal$, where for 
each $S \in \Ecal$ the corresponding $\hat{S}$ is not isomorphic to $T$.
In fact, if $\hat{S}$ would be isomorphic to $T$
then the isomorphism would take the partition $\Qcal$ of $Y$ to a partition $\tilde{\Qcal}$ of $X$.
Applying Lemma \ref{1-delta} to $\tilde{\Qcal}$ we see that there is some $N$
such that for all $n \geq N$ the conclusion of the lemma holds.
But since $S \in \Ecal$ this is a contradiction.
\end{proof}

\br

\section{The positive entropy theorem for amenable groups}\label{sec-A}

We fix an arbitrary infinite countable amenable group $G$.
We let $\mathbb{A}(G,\mu)$ denote the Polish space of measure preserving actions $\{T_g\}_{g \in G}$ of 
$G$ on the Lebesgue space $(X, \Xcal, \mu)$.
(For a description of the topology on $\mathbb{A}(G,\mu)$  we refer e.g. to \cite{Ke}.)

As in the proof of Theorem \ref{Thm-rB}
let $\Scal$ be the collection of Rokhlin cocycles from
$\Xb$ with values in MPT$(I, \la)$, that is, $\Scal$ is a family
$\{S^g\}_{g \in G}$, where each element $S^g$ is a 
collection of measurable maps $x \mapsto S^g_x \in $ MPT$(I, \la)$,
such that for $g, h \in G$ and $x \in X$ we have
$$
S^{gh}(x) = S^{g}(T_hx)S^h(x), \quad \mu \ {\text{a.e.}}.
$$
We associate to $S \in \Scal$ the {\em skew product transformation}
$$
\hat{S}^g(x,u) = (T_gx, S^g_x u),\quad  (x \in X, u \in I).
$$ 
Let $Y = X \times I$ and set $\Yb = (Y, \mathcal{Y}, \mu \times \la)$, with $\Ycal = \Xcal \otimes \Ccal$.

A free $G$-action $\Xb$ defines an equivalence relation $R \subset X \times X$, where
$(x , x') \in R$ iff \ $\exists g \in G, \ x' = gx$, and a cocycle $S \in \Scal$ defines uniquely a cocycle 
$\al$ on $R$ \footnote{A cocycle $\al$ on $R$ is a function 
from $R$ to MPT$(I,\la)$
which satisfies the {\em cocycle equation}
$$
\al(x,z) = \al(y,z)\al(x,y).
$$
}:
$$
\al(x,x') = S^g_x.
$$
This map is one-to-one and onto from the set of cocycles on $\Xb$
to the set of cocycles on $R$.
For more details on this correspondence see \cite[Section 20, C]{Ke}.

Let now 
$$
\Xb = (X, \Xcal, \mu, \{T_g\}_{g \in G}) \to \Xb_0 = 
(X_0, \Xcal_0, \mu_0, \{(T_0)_g\}_{g \in G})
$$ 
be a $G$-Bernoulli extension, where this notion is defined exactly as in Definition \ref{d:rB}.

\br

\begin{defn}
If $G$ and $H$ are two countable groups acting as measure preserving transformations 
$\{T_g\}_{g \in G}, \{S_h\}_{h \in H}$
on the measure space $(Z,\nu)$ we say that the actions are {\em orbit equivalent}
if for $\nu$-a.e. $z \in Z, \  Gz = Hz$.
\end{defn}

In \cite{OW-80} and \cite{CFW-81} it is shown that any ergodic measure preserving  action of an 
amenable group is orbit equivalent to an action of $\Z$.

\br

We can now state and prove an extension of Theorem \ref{Thm-rB} to free actions of $G$,
and moreover we can also get rid of the finite entropy assumption on $\Xb$.

\begin{thm}\label{Thm-rB-a}
Let $\Xb = (X, \Xcal, \mu,\{T_g\}_{g \in G})$ be an ergodic $G$-system which 
is relative Bernoulli over a free system $\Xb_0$ with finite relative entropy, so that $ \Xb = \Xb_0 \times \Xb_1$.
Then, the generic extension $\hat{S}$ of $\{T_g\}_{g \in G}$ is relatively Bernoulli over $\Xb_0$.
\end{thm}

\begin{proof}
By \cite{OW-80} there is $T_0 : X_0 \to X_0$ such that orbits of $T_0$ coincide with $G$-orbits on $X_0$,
and such that $T_0$ has zero entropy.
The $G$-factor map $\Xb = \Xb_0 \times  \Xb_1\to \Xb_0$ is given by a constant cocycle whose constant value is
the Bernoulli action on the Bernoulli factor $\Xb_1$.
We use this cocycle, now viewed as a cocycle on the equivalence relation defined by $T_0$,
to define an extension $T : X \to X$. By \cite{R-W} the relative entropy of a generic such 
$T$ over $T_0$ is the same as that of the $G$-action $\Xb$ over $\Xb_0$.
By \cite{D-P} the extension of $\Z$-systems $\pi : T \to T_0$ is again relatively Bernoulli.
Applying Theorem \ref{Thm-rB} to $\pi$ we conclude that a dense $G_\del$ subset $\Scal_1(\Z)$ of
extensions of $T$ is such that each $\hat{S} \in \Scal_1(\Z)$ is relatively Bernoulli over $T_0$.
Finally applying \cite{D-P} in the other direction we conclude that the  corresponding
set of extensions $\Scal_1(G)$ is again a dense $G_\del$ subset of $\Scal(G)$ and that for each $S
\in \Scal_1(G)$, the corresponding $G$-system is relatively Bernoulli over $\Xb_0$.
\end{proof}

As in the case of $\Z$-actions, with the same proof, we now obtain the following theorem.

\begin{thm}\label{Dam}
Every ergodic free $G$-system $\Xb$ of positive entropy is dominant.
\end{thm}

In view of Theorem \ref{thm-0e} the following question naturally arises:

\begin{prob}
Can Theorem \ref{thm-0e} be extended to all infinite countable amenable groups ?
\end{prob}

\br

\end{document}